\documentclass[12pt,reqno]{amsart}
\usepackage{amsfonts}
\usepackage{amssymb,amsmath}
\usepackage{amscd}
\usepackage{hyperref}
\usepackage{graphicx}
\usepackage{epsfig}
\usepackage{graphicx}
\usepackage{epstopdf}

\newtheorem{theorem}{Theorem}[section]

\newtheorem{example}{Example}[section]

\numberwithin{equation}{section}
\begin{document}
\title[Right Circulant Matrices  ]{Right Circulant Matrices with Generalized
Fibonacci and\ Lucas Polynomials and Coding Theory }
\author{S\"{U}MEYRA\ U\c{C}AR}
\address{Bal\i kesir University\\
Department of Mathematics\\
10145 Bal\i kesir, TURKEY}
\email{sumeyraucar@balikesir.edu.tr}
\author{Nihal YILMAZ\ \"{O}ZG\"{U}R}
\address{Bal\i kesir University\\
Department of Mathematics\\
10145 Bal\i kesir, TURKEY}
\email{nihal@balikesir.edu.tr}

\begin{abstract}
In this paper, we give two new coding algorithms by means of right circulant
matrices with elements generalized Fibonacci and Lucas polynomials. For this
purpose, we study basic properties of right circulant matrices using
generalized Fibonacci polynomials $F_{p,q,n}\left( x\right) $, generalized
Lucas polynomials $L_{p,q,n}\left( x\right) $ and geometric sequences.
\end{abstract}

\maketitle

\section{\textbf{Introduction}}

\label{sec:1}

There are many studies for the coding theory using Fibonacci numbers, Lucas
numbers, Lucas $p$ numbers, Pell numbers, generalized Pell $\left(
p,i\right) $-numbers in the literature $($for more details see \cite{prasad}%
, \cite{stakhov}, \cite{tas}, \cite{tas2}, \cite{ucar2} and references
therein$)$. Here we consider two classes of right circulant matrices whose
entries are generalized Fibonacci and Lucas polynomials to obtain new
coding/decoding algorithms.

Let $n>0$ be an integer. From \cite{catarino}, we know that a $g$-circulant
matrix is a square matrix of order n in the following form
\begin{equation*}
A_{g,n}=%
\begin{pmatrix}
a_{1} & a_{2} & \cdots & a_{n} \\
a_{n-g+1} & a_{n-g+2} & \cdots & a_{n-g} \\
a_{n-2g+1} & a_{n-2g+2} &  & a_{n-2g} \\
\vdots & \vdots & \ddots & \vdots \\
a_{g+1} & a_{g+2} & \cdots & a_{g}%
\end{pmatrix}%
,
\end{equation*}%
where g is a nonnegative integer and each of subscripts is understood to be
reduced modulo $n$. The first rows of $A_{g,n}$ is $(a_{1},a_{2},...,a_{n})$
and its $(j+1)$th row is obtained by giving its $j$th row a right circulant
shift by $g$ positions.

It is clear that $g=1$ or $g=n+1$ yields the standart right circulant
matrix, or simply, circulant matrix. Thus a right circulant matrix is
written as
\begin{equation*}
RCirc\left( a_{1},a_{2},...,a_{n}\right) =%
\begin{pmatrix}
a_{1} & a_{2} & \cdots & a_{n} \\
a_{n} & a_{1} & \cdots & a_{n-1} \\
\vdots & \vdots & \ddots & \vdots \\
a_{2} & a_{3} & \cdots & a_{1}%
\end{pmatrix}%
.
\end{equation*}%
A geometric sequence is known to be a sequence $\{a_{k}\}_{k=1}^{\infty }$
such that each term is given by a multiple of the previous one.

In \cite{catarino}, it was given a $g$-circulant, right circulant and left
circulant matrices whose entries are $h\left( x\right) $-Fibonacci
polynomials and presented the determinants of these matrices. In \cite%
{bueno-2016}, it was introduced the right circulant matrices with ratio of
the elements of Fibonacci and a geometric sequence and given eigenvalues,
determinants, Euclidean norms and inverses of these matrices.

In \cite{gong}, it has been dealt with circulant matrices with the
Jacobsthal and Jacobsthal-Lucas numbers, studied the invertibility of these
circulant matrices and presented the determinant and the inverse matrix.
Similarly, in \cite{shen}, it has been studied inverses and determinants of
the circulant matrices related to Fibonacci and Lucas numbers. Furthermore,
there are many applications of circulant matrices in the literature. For
example, these matrices has been studied related to models and several
differential equations such as fractional order models for nonlocal
epidemics, differential delay equations $($for more details one can see \cite%
{ahmed}, \cite{nussbaum}, \cite{qu}, \cite{wu} and the references therein$)$.

Recently, $h(x)$-Fibonacci polynomials are defined by $F_{h,0}(x)=0,$ $%
F_{h,1}(x)=1$ and $F_{h,n+1}(x)=h(x)F_{h,n}(x)+F_{h,n-1}(x)$ for $n\geq 1$,
and then it was given some properties of them in \cite{nalli}.

Let $p(x)$ and $q(x)$ be polynomials with real coefficients, $p\left(
x\right) \neq 0,q\left( x\right) \neq 0$ and $p^{2}\left( x\right) +4q\left(
x\right) >0.$ In \cite{kilic2011}, it was defined generalized Fibonacci
polynomials $F_{p,q,n}(x)$ by%
\begin{equation}
F_{p,q,n+1}(x)=p(x)F_{p,q,n}(x)+q(x)F_{p,q,n-1}(x)\text{, \ \ }n\geq 1
\label{eqn2}
\end{equation}%
with the initial values $F_{p,q,0}(x)=0,F_{p,q,1}(x)=1$ and generalized
Lucas polynomials $L_{p,q,n}(x)$ were defined by
\begin{equation}
L_{p,q,n+1}(x)=p(x)L_{p,q,n}(x)+q(x)L_{p,q,n-1}(x)\text{, \ \ }n\geq 1
\label{eqn10}
\end{equation}%
with the initial values $L_{p,q,0}(x)=2,L_{p,q,1}(x)=p\left( x\right) $. For
$p(x)=x$ and $q(x)=1$ we have Catalan's Fibonacci polynomials $F_{n}(x);$
for $p(x)=2x$ and $q(x)=1$ we have Byrd's polynomials $\varphi _{n}(x)$; $\ $%
for $p(x)=k$ and $q(x)=t$ we have generalized Fibonacci numbers $U_{n}$ ;
for $p(x)=k$ and $q(x)=1$ we have $k$-Fibonacci numbers $F_{k,n}$; for $%
p(x)=q(x)=1$ we have classical Fibonacci numbers $F_{n}$ (for more details
see \cite{catarino1}, \cite{falcon}, \cite{kalman}, \cite{koshy}, \cite{siar}
and \cite{ucar}$)$.

For $p(x)=x$ and $q(x)=1$ we have Lucas polynomials $L_{n}(x);$ for $p(x)=k$
and $q(x)=t$ we have generalized Lucas numbers $V_{n}$; for$\ p(x)=k$ and $%
q(x)=1$ we have $k$-Lucas numbers $L_{k,n}$; for $p(x)=$ $q(x)=1$ we have
classical Lucas numbers $L_{n}$ (for more details see \cite{falcon2007},
\cite{godase}, \cite{koshy}, \cite{siar} and \cite{ucar}$)$.

Let $\alpha \left( x\right) $ and $\beta \left( x\right) $ be the roots of
the characteristic equation
\begin{equation*}
v^{2}-vp\left( x\right) -q\left( x\right) =0
\end{equation*}%
of the recurrence relation $($\ref{eqn2}$).$ From \cite{kilic2011}, we know
that Binet formulas for generalized Fibonacci polynomials $F_{p,q,n}\left(
x\right) $ and generalized Lucas polynomials $L_{p,q,n}\left( x\right) $ are
of the following forms, respectively:
\begin{equation*}
F_{p,q,n}\left( x\right) =\frac{\alpha ^{n}\left( x\right) -\beta ^{n}\left(
x\right) }{\alpha \left( x\right) -\beta \left( x\right) },\text{ for }n\geq
0
\end{equation*}%
and%
\begin{equation*}
L_{p,q,n}\left( x\right) =\alpha ^{n}\left( x\right) +\beta ^{n}\left(
x\right) ,\text{ for }n\geq 0\text{,}
\end{equation*}%
where
\begin{equation*}
\alpha \left( x\right) =\frac{p\left( x\right) +\sqrt{p^{2}\left( x\right)
+4q\left( x\right) }}{2}\text{ and }\beta \left( x\right) =\frac{p\left(
x\right) -\sqrt{p^{2}\left( x\right) +4q\left( x\right) }}{2}.
\end{equation*}%
It is clear that $\alpha \left( x\right) +\beta \left( x\right) =p\left(
x\right) ,$ $\alpha \left( x\right) \beta \left( x\right) =-q\left( x\right)
$ and $\alpha \left( x\right) -\beta \left( x\right) =\sqrt{p^{2}\left(
x\right) +4q\left( x\right) }.$

In this study, we investigate right circulant matrices using generalized
Fibonacci polynomials, generalized Lucas polynomials and geometric
sequences. In Section \ref{sec:2}, we give the eigenvalues and determinants
of the right circulant matrices whose entries are the ratio of the elements
of generalized Fibonacci sequence and some geometric sequences. In Section %
\ref{sec:3}, we give right circulant matrices whose entries are the
generalized Fibonacci and Lucas polynomials and calculate the determinants
of these matrices. In Section \ref{sec:4}, we give some applications of
right circulant matrices to coding theory.

From now on, we shortly denote $\alpha \left( x\right) $ by $\alpha $, $%
\beta \left( x\right) $ by $\beta $, $p\left( x\right) $ by $p$ and $q\left(
x\right) $ by $q.$

\section{\textbf{Right Circulant Matrices with} \textbf{Generalized
Fibonacci Polynomials and Geometric Sequences}}

\label{sec:2}

Let $\{f_{n}\}_{n=1}^{\infty }$ be the sequence of the form%
\begin{equation*}
f_{n}=\frac{F_{p,q,n}(x)}{ar^{n}},
\end{equation*}%
where $F_{p,q,n}(x)$ is the $n$-th generalized Fibonacci polynomial and $%
ar^{n}$ is the $n$-th element of any geometric sequence.

Using these types of sequences, we consider the following right circulant
matrix $\mathcal{F}_{n}$:
\begin{equation*}
\mathcal{F}_{n}=%
\begin{pmatrix}
f_{0} & f_{1} & f_{2} & \cdots & f_{n-2} & f_{n-1} \\
f_{n-1} & f_{0} & f_{1} & \cdots & f_{n-3} & f_{n-2} \\
f_{n-2} & f_{n-1} & f_{0} & \cdots & f_{n-4} & f_{n-3} \\
\vdots & \vdots & \vdots & \ddots & \vdots & \vdots \\
f_{2} & f_{3} & f_{4} & \cdots & f_{0} & f_{1} \\
f_{1} & f_{2} & f_{3} & \cdots & f_{n-1} & f_{0}%
\end{pmatrix}%
.
\end{equation*}

\begin{theorem}
\label{thm1}The eigenvalues of the matrix $\mathcal{F}_{n}$ are as follows:%
\begin{equation*}
\lambda _{m}=\frac{-rF_{p,q,n}(x)-w^{-m}\left( qF_{p,q,n-1}(x)-r^{n}\right)
}{ar^{n-1}\left( r-\alpha w^{-m}\right) \left( r-\beta w^{-m}\right) },
\end{equation*}%
where $m=0,1,...,n-1,$ $\alpha =\frac{p+\sqrt{p^{2}+4q}}{2},\beta =\frac{p-%
\sqrt{p^{2}+4q}}{2}$ and $w=e^{\frac{2\pi i}{n}}$.
\end{theorem}

\begin{proof}
From \cite{bueno-2012}, we know that the eigenvalues of a right circulant
matrix $\mathcal{F}_{n}$ are
\begin{equation}
\lambda _{m}=\overset{n-1}{\underset{k=0}{\sum }}f_{k}w^{-mk},  \label{eqn1}
\end{equation}%
where $m=e^{\frac{2\pi i}{n}}$ and $m=0,1,2,...,n-1.$ Using the equation $($%
\ref{eqn1}$)$ and the Binet's formula for the generalized Fibonacci
polynomials $F_{p,q,n}(x),$ we get \
\begin{equation*}
\lambda _{m}=\overset{n-1}{\underset{k=0}{\sum }}\frac{F_{p,q,k}(x)}{ar^{k}}%
w^{-mk}=\overset{n-1}{\underset{k=0}{\sum }}\frac{\alpha ^{k}-\beta ^{k}}{%
\left( \alpha -\beta \right) ar^{k}}w^{-mk},
\end{equation*}%
where $\alpha =\frac{p+\sqrt{p^{2}+4q}}{2},\beta =\frac{p-\sqrt{p^{2}+4q}}{2}%
.$ Then using the properties of the geometric series, we obtain%
\begin{eqnarray*}
\lambda _{m} &=&\frac{1}{a\left( \alpha -\beta \right) }\left( \frac{%
1-\left( \alpha /r\right) ^{n}}{1-\alpha w^{-m}/r}-\frac{1-\left( \beta
/r\right) ^{n}}{1-\beta w^{-m}/r}\right) \\
&=&\frac{1}{ar^{n-1}\left( \alpha -\beta \right) }\left( \frac{r^{n}-\alpha
^{n}}{r-\alpha w^{-m}}-\frac{r^{n}-\beta ^{n}}{r-\beta w^{-m}}\right) \\
&=&\frac{1}{ar^{n-1}\left( \alpha -\beta \right) }\left( \frac{\left(
r^{n}-\alpha ^{n}\right) \left( r-\beta w^{-m}\right) -\left( r-\alpha
w^{-m}\right) \left( r^{n}-\beta ^{n}\right) }{\left( r-\alpha w^{-m}\right)
\left( r-\beta w^{-m}\right) }\right) .
\end{eqnarray*}

By the fact $\alpha \beta =-q$, we find%
\begin{eqnarray*}
\lambda _{m} &=&\frac{-r\left( \alpha ^{n}-\beta ^{n}\right)
+r^{n}w^{-m}\left( \alpha -\beta \right) -w^{-m}\left( -\beta \alpha
^{n}+\alpha \beta ^{n}\right) }{ar^{n-1}\left( \alpha -\beta \right) \left(
r-\alpha w^{-m}\right) \left( r-\beta w^{-m}\right) } \\
&=&\frac{-r\left( \alpha ^{n}-\beta ^{n}\right) +r^{n}w^{-m}\left( \alpha
-\beta \right) -w^{-m}q\left( \alpha ^{n-1}-\beta ^{n-1}\right) }{%
ar^{n-1}\left( \alpha -\beta \right) \left( r-\alpha w^{-m}\right) \left(
r-\beta w^{-m}\right) } \\
&=&\frac{-rF_{p,q,n}(x)+r^{n}w^{-m}F_{p,q,1}(x)-w^{-m}qF_{p,q,n-1}(x)}{%
ar^{n-1}\left( r-\alpha w^{-m}\right) \left( r-\beta w^{-m}\right) }
\end{eqnarray*}%
and so
\begin{equation*}
\lambda _{m}=\frac{-rF_{p,q,n}(x)-w^{-m}\left( qF_{p,q,n-1}(x)-r^{n}\right)
}{ar^{n-1}\left( r-\alpha w^{-m}\right) \left( r-\beta w^{-m}\right) }.
\end{equation*}
\end{proof}

Now we give the following theorem.

\begin{theorem}
\ \label{thm2}The determinant of the matrix $\mathcal{F}_{n}$ is
\begin{equation*}
\det (\mathcal{F}_{n})=\frac{\left( -1\right)
^{n}r^{n}F_{p,q,n}^{n}(x)-\left( qF_{p,q,n-1}(x)-r^{n}\right) ^{n}}{%
a^{n}r^{n(n-1)}-\left( r^{2n}-r^{n}L_{p,q,n}(x)+(-q)^{n}\right) }.
\end{equation*}
\end{theorem}

\begin{proof}
Since the determinant of a matrix is the product of its eigenvalues, by
Theorem \ref{thm1} we obtain%
\begin{equation}
\det (\mathcal{F}_{n})=\prod_{m=0}^{n-1}\frac{-rF_{p,q,n}(x)-w^{-m}\left(
qF_{p,q,n-1}(x)-r^{n}\right) }{ar^{n-1}\left( r-\alpha w^{-m}\right) \left(
r-\beta w^{-m}\right) }.  \label{eqn3}
\end{equation}%
From the complex analysis, we know
\begin{equation}
\prod_{m=0}^{n-1}\left( x-yw^{-m}\right) =x^{n}-y^{n}  \label{eqn4}
\end{equation}%
Applying the equation $($\ref{eqn4}$)$\ to the equation $($\ref{eqn3}$),$ we
find%
\begin{eqnarray*}
\det (\mathcal{F}_{n}) &=&\frac{\left( -1\right)
^{n}r^{n}F_{p,q,n}^{n}(x)-\left( qF_{p,q,n-1}(x)-r^{n}\right) ^{n}}{%
a^{n}r^{n\left( n-1\right) }\left( r^{n}-\alpha ^{n}\right) \left(
r^{n}-\beta ^{n}\right) } \\
&=&\frac{\left( -1\right) ^{n}r^{n}F_{p,q,n}^{n}(x)-\left(
qF_{p,q,n-1}(x)-r^{n}\right) ^{n}}{a^{n}r^{n\left( n-1\right) }\left(
r^{2n}-r^{n}\left( \alpha ^{n}+\beta ^{n}\right) +\left( -q\right)
^{n}\right) }.
\end{eqnarray*}%
Using the Binet's formulas for the generalized Lucas polynomials $%
L_{p,q,n}\left( x\right) ,$ we get%
\begin{equation*}
\det (\mathcal{F}_{n})=\frac{\left( -1\right)
^{n}r^{n}F_{p,q,n}^{n}(x)-\left( qF_{p,q,n-1}(x)-r^{n}\right) ^{n}}{%
a^{n}r^{n\left( n-1\right) }\left( r^{2n}-r^{n}L_{p,q,n}(x)+\left( -q\right)
^{n}\right) }.
\end{equation*}
\end{proof}

Notice that for $p=x$ and $q=1$; for $p=2x$ and $q=1$; $p=k$ and $q=t$; $p=k$
and $q=1$; $p=q=1$ in Theorem \ref{thm1} and Theorem \ref{thm2}, we have
similar theorems for Catalan's Fibonacci polynomials $F_{n}\left( x\right) $%
, Byrd's polynomials $\varphi _{n}(x)$, generalized Fibonacci numbers $U_{n}$%
, $k$-Fibonacci numbers $F_{k,n}$, classical Fibonacci numbers $F_{n}$,
respectively. Also, in \cite{bueno-2012-2}, the right circulant matrix with
Fibonacci sequence was defined and given eigenvalues, Euclidean norm of this
matrix.

\section{\textbf{Right Circulant Matrices with Generalized Fibonacci and
Lucas Polynomials}}

\label{sec:3}

In this section we give the determinant of a right circulant matrix whose
elements are generalized Fibonacci polynomials $F_{p,q,n}\left( x\right) $
and generalized Lucas polynomials $L_{p,q,n}\left( x\right) $.

\begin{theorem}
\label{thm3}Let $G_{n}$ be a right circulant matrix of the following form:%
\begin{equation}
G_{n}=\left(
\begin{array}{cccccc}
F_{p,q,1}(x) & F_{p,q,2}(x) & F_{p,q,3}(x) & \cdots & F_{p,q,n-1}(x) &
F_{p,q,n}(x) \\
F_{p,q,n}(x) & F_{p,q,1}(x) & F_{p,q,2}(x) & \cdots & F_{p,q,n-2}(x) &
F_{p,q,n-1}(x) \\
F_{p,q,n-1}(x) & F_{p,q,n}(x) & F_{p,q,1}(x) & \cdots & F_{p,q,n-3}(x) &
F_{p,q,n-2}(x) \\
\vdots & \vdots & \vdots & \ddots & \vdots & \vdots \\
F_{p,q,4}(x) & F_{p,q,5}(x) & F_{p,q,6}(x) & \cdots & F_{p,q,2}(x) &
F_{p,q,3}(x) \\
F_{p,q,3}(x) & F_{p,q,4}(x) & F_{p,q,5}(x) & \cdots & F_{p,q,1}(x) &
F_{p,q,2}(x) \\
F_{p,q,2}(x) & F_{p,q,3}(x) & F_{p,q,4}(x) & \cdots & F_{p,q,n}(x) &
F_{p,q,1}(x)%
\end{array}%
\right) .  \label{eqn12}
\end{equation}%
Then we have
\begin{equation}
\det (G_{n})=\left( 1-F_{p,q,n+1}(x)\right) ^{n-1}+\left(
qF_{p,q,n}(x)\right) ^{n-2}\overset{n-1}{\underset{k=1}{\sum }}\left( \frac{%
1-F_{p,q,n+1}(x)}{qF_{p,q,n}(x)}\right) ^{k-1}qF_{p,q,k}(x).  \label{eqn6}
\end{equation}
\end{theorem}

\begin{proof}
For $n=1$, det$(G_{1})=1$ satisfies the equation $($\ref{eqn6}$).$ Let us
consider the case $n\geq 2.$ Consider the following matrices:
\begin{equation}
A=\left(
\begin{array}{cccccccc}
1 & 0 & 0 & 0 & \cdots & 0 & 0 & 0 \\
-p & 0 & 0 & 0 & \cdots & 0 & 0 & 1 \\
-q & 0 & 0 & 0 & \cdots & 0 & 1 & -p \\
\vdots & \vdots & \vdots &  & \ddots & \vdots & \vdots &  \\
0 & 0 & 0 & 1 & \cdots & 0 & 0 & 0 \\
0 & 0 & 1 & -p & \cdots & 0 & 0 & 0 \\
0 & 1 & -p & -q & \cdots & 0 & 0 & 0%
\end{array}%
\right) _{n\times n}  \label{eqn11}
\end{equation}%
and%
\begin{equation*}
B=\left(
\begin{array}{cccccc}
1 & 0 & 0 & \cdots & 0 & 0 \\
0 & \left( \frac{qF_{p,q,n}\left( x\right) }{1-F_{p,q,n+1}\left( x\right) }%
\right) ^{n-2} & 0 & \cdots & 0 & 1 \\
0 & \left( \frac{qF_{p,q,n}\left( x\right) }{1-F_{p,q,n+1}\left( x\right) }%
\right) ^{n-3} & 0 & \cdots & 1 & 0 \\
\vdots & \vdots & \vdots & \ddots & \vdots &  \\
0 & \left( \frac{qF_{p,q,n}\left( x\right) }{1-F_{p,q,n+1}\left( x\right) }%
\right) & 1 & \cdots & 0 & 0 \\
0 & 1 & 0 & \cdots & 0 & 0%
\end{array}%
\right) _{n\times n}.
\end{equation*}%
Notice that
\begin{equation*}
\det (A)=\det (B)=\left( -1\right) ^{\frac{\left( n-1\right) \left(
n-2\right) }{2}}.
\end{equation*}%
If we multiply the matrices $A,G_{n}$ and $B$ we have the following product
matrices:

\begin{equation*}
{\footnotesize AG_{n}B=%
\begin{pmatrix}
F_{p,q,1}\left( x\right) & \alpha _{n} & F_{p,q,n-1}\left( x\right) & \cdots
& F_{p,q,3}\left( x\right) & F_{p,q,2}\left( x\right) \\
0 & \beta _{n} & F_{p,q,n-2}\left( x\right) & \cdots & qF_{p,q,2}\left(
x\right) & qF_{p,q,1}\left( x\right) \\
0 & 0 & F_{p,q,1}\left( x\right) -F_{p,q,n+1}\left( x\right) & \cdots &  &
\\
0 & 0 & -qF_{p,q,n}\left( x\right) &  &  &  \\
\vdots & \vdots &  & \ddots &  &  \\
0 & 0 &  &  & \ddots &  \\
0 & 0 & 0 &  & -qF_{p,q,n}\left( x\right) & F_{p,q,1}\left( x\right)
-F_{p,q,n+1}\left( x\right)%
\end{pmatrix}%
,}
\end{equation*}%
where
\begin{equation}
\alpha _{n,p,q}=\overset{n-1}{\underset{k=1}{\sum }}\left( \frac{%
qF_{p,q,n}(x)}{F_{p,q,1}(x)-F_{p,q,n+1}(x)}\right) ^{n-\left( k+1\right)
}F_{p,q,k+1}(x)  \label{eqn7}
\end{equation}%
and%
\begin{equation}
\beta _{n,p,q}=\left( 1-F_{p,q,n+1}\left( x\right) \right) +\overset{n-1}{%
\underset{k=1}{\sum }}\left( \frac{qF_{p,q,n}(x)}{F_{p,q,1}(x)-F_{p,q,n+1}(x)%
}\right) ^{n-\left( k+1\right) }qF_{p,q,k}\left( x\right) .  \label{eqn8}
\end{equation}%
Then we have%
\begin{equation*}
\det \left( AG_{n}B\right) =F_{p,q,1}\left( x\right) \beta _{n,p,q}\left(
F_{p,q,1}\left( x\right) -F_{p,q,n+1}\left( x\right) \right) ^{n-2}.
\end{equation*}%
Using the equation $($\ref{eqn8}$),$ we get
\begin{eqnarray*}
\det \left( AG_{n}B\right) &=&\left( 1-F_{p,q,n+1}\left( x\right) \right)
^{n-1} \\
&&+\left( qF_{p,q,n}\left( x\right) \right) ^{n-2}\overset{n-1}{\underset{k=1%
}{\sum }}\left( \frac{1-F_{p,q,n+1}(x)}{qF_{p,q,n}(x)}\right)
^{k-1}qF_{p,q,k}\left( x\right) .
\end{eqnarray*}%
Since $\det \left( AG_{n}B\right) =\det \left( G_{n}\right) ,$ we find%
\begin{eqnarray*}
\det \left( G_{n}\right) &=&\left( 1-F_{p,q,n+1}\left( x\right) \right)
^{n-1} \\
&&+\left( qF_{p,q,n}\left( x\right) \right) ^{n-2}\overset{n-1}{\underset{k=1%
}{\sum }}\left( \frac{1-F_{p,q,n+1}(x)}{qF_{p,q,n}(x)}\right)
^{k-1}qF_{p,q,k}\left( x\right) .
\end{eqnarray*}
\end{proof}

Now we give the following theorem for generalized Lucas polynomials $%
L_{p,q,n}\left( x\right) $.

\begin{theorem}
Let $H_{n}$ be a right circulant matrix of the following form:%
\begin{equation}
H_{n}=\left(
\begin{array}{cccccc}
L_{p,q,1}(x) & L_{p,q,2}(x) & L_{p,q,3}(x) & \cdots & L_{p,q,n-1}(x) &
L_{p,q,n}(x) \\
L_{p,q,n}(x) & L_{p,q,1}(x) & L_{p,q,2}(x) & \cdots & L_{p,q,n-2}(x) &
L_{p,q,n-1}(x) \\
L_{p,q,n-1}(x) & L_{p,q,n}(x) & L_{p,q,1}(x) & \cdots & L_{p,q,n-3}(x) &
L_{p,q,n-2}(x) \\
\vdots & \vdots & \vdots & \ddots & \vdots & \vdots \\
L_{p,q,4}(x) & L_{p,q,5}(x) & L_{p,q,6}(x) & \cdots & L_{p,q,2}(x) &
L_{p,q,3}(x) \\
L_{p,q,3}(x) & L_{p,q,4}(x) & L_{p,q,5}(x) & \cdots & L_{p,q,1}(x) &
L_{p,q,2}(x) \\
L_{p,q,2}(x) & L_{p,q,3}(x) & L_{p,q,4}(x) & \cdots & L_{p,q,n}(x) &
L_{p,q,1}(x)%
\end{array}%
\right) .  \label{eqn13}
\end{equation}%
Then we have
\begin{equation}
\begin{array}{l}
\det (H_{n})=L_{p,q,1}\left( x\right) \left( L_{p,q,1}\left( x\right)
-L_{p,q,n+1}(x)\right) ^{n-1} \\
+L_{p,q,1}\left( x\right) q^{n-1}\left( L_{p,q,n}(x)-2\right) ^{n-2}\overset{%
n-1}{\underset{k=1}{\sum }}\left( \frac{L_{p,q,1}\left( x\right)
-L_{p,q,n+1}(x)}{qL_{p,q,n}(x)-2q}\right) ^{k-1}L_{p,q,k}(x) \\
-2q^{n-1}\left( L_{p,q,n}(x)-2\right) ^{n-2}\overset{n-1}{\underset{k=1}{%
\sum }}\left( \frac{L_{p,q,1}\left( x\right) -L_{p,q,n+1}(x)}{%
qL_{p,q,n}(x)-2q}\right) ^{k-1}L_{p,q,k+1}(x).%
\end{array}
\label{eqn9}
\end{equation}
\end{theorem}

\begin{proof}
For $n=1$, det$(H_{1})=p$ satisfies the equation $($\ref{eqn9}$).$ Let us
consider the case $n\geq 2.$ Let $A$ be a matrix of the form given in $($\ref%
{eqn11}$)$ and D be a matrix of the following form:%
\begin{equation*}
D=\left(
\begin{array}{cccccc}
1 & 0 & 0 & \cdots & 0 & 0 \\
0 & \left( \frac{qL_{p,q,n}\left( x\right) -2q}{L_{p,q,1}\left( x\right)
-L_{p,q,n+1}\left( x\right) }\right) ^{n-2} & 0 & \cdots & 0 & 1 \\
0 & \left( \frac{qL_{p,q,n}\left( x\right) -2q}{L_{p,q,1}\left( x\right)
-L_{p,q,n+1}\left( x\right) }\right) ^{n-3} & 0 & \cdots & 1 & 0 \\
\vdots & \vdots & \vdots & \ddots & \vdots &  \\
0 & \left( \frac{qL_{p,q,n}\left( x\right) -2q}{L_{p,q,1}\left( x\right)
-L_{p,q,n+1}\left( x\right) }\right) & 1 & \cdots & 0 & 0 \\
0 & 1 & 0 & \cdots & 0 & 0%
\end{array}%
\right) _{n\times n}.
\end{equation*}%
Using the properties of determinants and multiplying these matrices $A,H_{n}$
and $D$ the proof can be completed by similar arguments used in the proof of
the above theorem.
\end{proof}

\section{An Application to Coding Theory}

\label{sec:4}

In this section, we give two new coding/decoding methods using the right
circulant matrices $G_{n}$ and $H_{n}$ for $p=q=1$. At first, we give an
algorithm by means of the generalized Fibonacci polynomials. Following the
notations in \cite{tas}, we give generalized Fibonacci and Lucas blocking
algorithms with transformations%
\begin{equation*}
M\times G_{n}=E\text{, }M\times H_{n}=E
\end{equation*}%
and%
\begin{equation*}
E\times \left( G_{n}\right) ^{-1}=M\text{, }E\times \left( H_{n}\right)
^{-1}=M\text{,}
\end{equation*}%
where $M$ is nonsingular square message matrix, $E$ is a code matrix, $G_{n}$
is coding matrix and the inverse matrix $\left( G_{n}\right) ^{-1}$ is
decoding matrix.

We put our message in a matrix adding zero between two words and end of the
message until we obtain the size of the message matrix is $3m$. Dividing the
message square matrix $M$ into the block matrices, named $B_{i}$ ($1\leq
i\leq m^{2}$), of size $3\times 3$, from left to right, we can construct a
new coding method.

Now we explain the symbols of our coding method. Suppose that matrices $%
B_{i} $ and $E_{i}$ are of the following forms:%
\begin{equation*}
B_{i}=\left[
\begin{array}{ccc}
b_{1}^{i} & b_{2}^{i} & b_{3}^{i} \\
b_{4}^{i} & b_{5}^{i} & b_{6}^{i} \\
b_{7}^{i} & b_{8}^{i} & b_{9}^{i}%
\end{array}%
\right] \text{and }E_{i}=\left[
\begin{array}{ccc}
e_{1}^{i} & e_{2}^{i} & e_{3}^{i} \\
e_{4}^{i} & e_{5}^{i} & e_{6}^{i} \\
e_{7}^{i} & e_{8}^{i} & e_{9}^{i}%
\end{array}%
\right] \text{.}
\end{equation*}%
We use the matrix $G_{n}$ given in $($\ref{eqn12}$)$ for $p=q=1$ and rewrite
the elements of this matrix as $G_{n}=\left[
\begin{array}{ccc}
g_{1} & g_{2} & g_{3} \\
g_{3} & g_{1} & g_{2} \\
g_{2} & g_{3} & g_{1}%
\end{array}%
\right] .$ The number of the block matrices $B_{i}$ is denoted by $b$. In
accordance with $b$, we choose the number $n$ as follows:%
\begin{equation*}
n=\left\{
\begin{array}{ccc}
3 & \text{,} & b=1 \\
3b & \text{,} & b\neq 1%
\end{array}%
\right. .
\end{equation*}%
Using the chosen $n$, we write the following character table according to $%
mod27$ (this table can be enlarged according to the used characters in the
message matrix). We begin the \textquotedblleft $n$\textquotedblright\ for
the first character.%
\begin{equation}
\begin{tabular}{|c|c|c|c|c|c|c|c|c|}
\hline
A & B & C & D & E & F & G & H & I \\ \hline
$n$ & $n+1$ & $n+2$ & $n+3$ & $n+4$ & $n+5$ & $n+6$ & $n+7$ & $n+8$ \\ \hline
J & K & L & M & N & O & P & Q & R \\ \hline
$n+9$ & $n+10$ & $n+11$ & $n+12$ & $n+13$ & $n+14$ & $n+15$ & $n+16$ & $n+17$
\\ \hline
S & T & U & V & W & X & Y & Z & 0 \\ \hline
$n+18$ & $n+19$ & $n+20$ & $n+21$ & $n+22$ & $n+23$ & $n+24$ & $n+25$ & $%
n+26 $ \\ \hline
\end{tabular}
\label{eqn14}
\end{equation}

\textbf{Generalized Fibonacci Blocking Algorithm}

\textbf{Coding Algorithm}

\textbf{Step 1.} Divide the matrix $M$ into blocks $B_{i}$ $\left( 1\leq
i\leq m^{2}\right) $.

\textbf{Step 2.} Choose $n$.

\textbf{Step 3. }Determine $b_{j}^{i}$ $\left( 1\leq j\leq 9\right) $.

\textbf{Step 4.} Compute $\det (B_{i})\rightarrow d_{i}$.

\textbf{Step 5.} Construct $K=\left[ d_{i},b_{k}^{i}\right] _{k\in
\{1,2,3,4,6,7,8,9\}}$.

\textbf{Step 6. }End of algorithm.

\textbf{Decoding Algorithm }

\textbf{Step 1.} Compute $G_{n}$.

\textbf{Step 2.} Determine $g_{j}$ $(1\leq j\leq 3)$.

\textbf{Step 3. }Compute $g_{1}b_{1}^{i}+g_{3}b_{2}^{i}+g_{2}b_{3}^{i}%
\rightarrow e_{1}^{i},$ $\left( 1\leq i\leq m^{2}\right) $.

\qquad \qquad \qquad \qquad $g_{2}b_{1}^{i}+g_{1}b_{2}^{i}+g_{3}b_{3}^{i}%
\rightarrow e_{2}^{i},$

\qquad \qquad \qquad \qquad $g_{3}b_{1}^{i}+g_{2}b_{2}^{i}+g_{1}b_{3}^{i}%
\rightarrow e_{3}^{i},$

\qquad \qquad \qquad \qquad $g_{1}b_{7}^{i}+g_{3}b_{8}^{i}+g_{2}b_{9}^{i}%
\rightarrow e_{7}^{i},$

\qquad \qquad \qquad \qquad $g_{2}b_{7}^{i}+g_{1}b_{8}^{i}+g_{3}b_{9}^{i}%
\rightarrow e_{8}^{i},$

\qquad \qquad \qquad \qquad $g_{3}b_{7}^{i}+g_{2}b_{8}^{i}+g_{1}b_{9}^{i}%
\rightarrow e_{9}^{i}.$

\textbf{Step 4.} Solve%
\begin{eqnarray*}
\det \left( G_{3}\right) \times d_{i}
&=&e_{1}^{i}e_{9}^{i}(g_{2}b_{4}^{i}+g_{1}x_{i}+g_{3}b_{6}^{i})+e_{8}^{i}e_{3}^{i}(g_{1}b_{4}^{i}+g_{3}x_{i}+g_{2}b_{6}^{i})
\\
&&+e_{7}^{i}e_{2}^{i}(g_{3}b_{4}^{i}+g_{2}x_{i}+g_{1}b_{6}^{i})-(e_{3}^{i}e_{7}^{i}(g_{2}b_{4}^{i}+g_{1}x_{i}+g_{3}b_{6}^{i})
\\
&&+e_{8}^{i}e_{1}^{i}(g_{3}b_{4}^{i}+g_{2}x_{i}+g_{1}b_{6}^{i})+e_{9}^{i}e_{2}^{i}(g_{1}b_{4}^{i}+g_{3}x_{i}+g_{2}b_{6}^{i}).
\end{eqnarray*}

\textbf{Step 5. }Substitute for $x_{i}=b_{5}^{i}$.

\textbf{Step 6.} Construct $B_{i}$.

\textbf{Step 7.} Construct $M$.

\textbf{Step 8.} End of algorithm.

We give an application of the above generalized Fibonacci blocking algorithm
in the following example for $b=1$.

\begin{example}
\label{exm1}Let us consider the message matrix for the following message text%
$:$%
\begin{equation*}
\text{\textquotedblleft SUMEYRA\textquotedblright }
\end{equation*}%
Using the message text, we get the following message matrix $M:$%
\begin{equation*}
M=\left[
\begin{array}{ccc}
S & U & M \\
E & Y & R \\
A & 0 & 0%
\end{array}%
\right] _{3\times 3}.
\end{equation*}%
\textbf{Coding Algorithm:}

\textbf{Step 1. }We construct the message text $M$ of size $3\times 3$,
named $B_{1}:$%
\begin{equation*}
B_{1}=\left[
\begin{array}{ccc}
S & U & M \\
E & Y & R \\
A & 0 & 0%
\end{array}%
\right] \text{.}
\end{equation*}%
\textbf{Step 2.} Since $b=1$, we calculate $n=3$. For $n=3$, we use the
following \textquotedblleft letter table\textquotedblright\ for the message
matrix $M:$%
\begin{equation*}
\begin{tabular}{|l|l|l|l|l|l|l|l|}
\hline
$S$ & $U$ & $M$ & $E$ & $Y$ & $R$ & $A$ & $0$ \\ \hline
$21$ & $23$ & $15$ & $7$ & $27$ & $20$ & $3$ & $2$ \\ \hline
\end{tabular}%
.
\end{equation*}%
\textbf{Step 3.} We have the elements of the block $B_{1}$ as follows:%
\begin{equation*}
\begin{tabular}{|l|l|l|}
\hline
$b_{1}^{1}=21$ & $b_{2}^{1}=23$ & $b_{3}^{1}=15$ \\ \hline
$b_{4}^{1}=7$ & $b_{5}^{1}=27$ & $b_{6}^{1}=20$ \\ \hline
$b_{7}^{1}=3$ & $b_{8}^{1}=2$ & $b_{9}^{1}=2$ \\ \hline
\end{tabular}%
.
\end{equation*}%
\textbf{Step 4.} Now we calculate the determinant $d_{1}$ of the block $%
B_{1}:$%
\begin{equation*}
\begin{tabular}{|l|}
\hline
$d_{1}=\det (B_{1})=347$ \\ \hline
\end{tabular}%
.
\end{equation*}%
\textbf{Step 5.} Using Step 3 and Step 4, we obtain the following matrix $K:$%
\begin{equation*}
K=\left[
\begin{array}{ccccccccc}
347 & 21 & 23 & 15 & 7 & 20 & 3 & 2 & 2%
\end{array}%
\right] .
\end{equation*}%
\textbf{Step 6.} End of algorithm.\newline
\textbf{Decoding algorithm:}

\textbf{Step 1.} By $($\ref{eqn12}$),$ we know that%
\begin{equation*}
G_{3}=\left[
\begin{array}{ccc}
1 & 1 & 2 \\
2 & 1 & 1 \\
1 & 2 & 1%
\end{array}%
\right] \text{.}
\end{equation*}%
\textbf{Step 2. }The elements of $G_{3}$ are denoted by%
\begin{equation*}
g_{1}=1\text{, }g_{2}=1\text{ and }g_{3}=2\text{.}
\end{equation*}%
\textbf{Step 3.} We compute the elements $%
e_{1}^{1},e_{2}^{1},e_{3}^{1},e_{4}^{1},e_{7}^{1},e_{8}^{1},e_{9}^{1}$ to
construct the matrix $E_{1}:$%
\begin{equation*}
e_{1}^{1}=82\text{, }e_{2}^{1}=74\text{, }e_{3}^{1}=80\text{, }e_{7}^{1}=9%
\text{, }e_{8}^{1}=9\text{ and }e_{9}^{1}=10\text{.}
\end{equation*}%
\textbf{Step 4.} We calculate the elements $x_{1}:$%
\begin{eqnarray*}
4\times 347 &=&80624+2926x_{1}-78912-2938x_{1} \\
&\Rightarrow &x_{1}=27\text{.}
\end{eqnarray*}%
\textbf{Step 5.} We rename $x_{1}$ as follows$:$%
\begin{equation*}
x_{1}=b_{5}^{1}=27\text{.}
\end{equation*}%
\textbf{Step 6. }We construct the block matrix $B_{1}:$%
\begin{equation*}
B_{1}=\left[
\begin{array}{ccc}
21 & 23 & 15 \\
7 & 27 & 20 \\
3 & 2 & 2%
\end{array}%
\right] .
\end{equation*}%
\textbf{Step 7.} We obtain the message matrix $M:$%
\begin{equation*}
M=\left[
\begin{array}{ccc}
21 & 23 & 15 \\
7 & 27 & 20 \\
3 & 2 & 2%
\end{array}%
\right] =\left[
\begin{array}{ccc}
S & U & M \\
E & Y & R \\
A & 0 & 0%
\end{array}%
\right] .
\end{equation*}%
\textbf{Step 8.} End of algorithm.
\end{example}

Now, we give another blocking algorithm by means of the generalized Lucas
polynomials $L_{p,q,n}\left( x\right) $. Let's suppose
\begin{equation*}
B_{i}=\left[
\begin{array}{cc}
b_{1}^{i} & b_{2}^{i} \\
b_{3}^{i} & b_{4}^{i}%
\end{array}%
\right] \text{ and }E_{i}=\left[
\begin{array}{cc}
e_{1}^{i} & e_{2}^{i} \\
e_{3}^{i} & e_{4}^{i}%
\end{array}%
\right] .
\end{equation*}%
We use the matrix $H_{n}$ given in $($\ref{eqn13}$)$ for $p=q=1$ and we
rewrite the elements of this matrix as $H_{n}=\left[
\begin{array}{cc}
h_{1} & h_{2} \\
h_{2} & h_{1}%
\end{array}%
\right] .$ Similarly, the number of the block matrices $B_{i}$ is denoted by
$b$. According to $b$, we choose the number $n$ as follows:
\begin{equation*}
n=\left\{
\begin{array}{ccc}
2 & \text{,} & b=1 \\
2b & \text{,} & b\neq 1%
\end{array}%
\right. .
\end{equation*}%
Using the chosen $n$, we write the character table given in $($\ref{eqn14}$)$
according to $mod27$ or we can differently array this table. For example, we
begin the $"n"$ for the first, second,central, last character etc.

\textbf{Generalized Lucas Blocking Algorithm \ \ }

\textbf{Coding Algorithm \ \ \ \ \ \ \ \ \ }

\textbf{Step 1.} Divide the matrix $M$ into blocks $B_{i}$ $\left( 1\leq
i\leq m^{2}\right) $.

\textbf{Step 2.} Choose $n$.

\textbf{Step 3. }Determine $b_{j}^{i}$ $\left( 1\leq j\leq 4\right) $.

\textbf{Step 4.} Compute $\det (B_{i})\rightarrow d_{i}$.

\textbf{Step 5.} Construct $K=\left[ d_{i},b_{k}^{i}\right] _{k\in
\{1,3,4\}} $.

\textbf{Step 6. }End of algorithm.

\textbf{Decoding Algorithm }

\textbf{Step 1.} Compute $H_{n}$.

\textbf{Step 2.} Determine $h_{j}$ $(1\leq j\leq 2)$.

\textbf{Step 3. }Compute $h_{1}b_{3}^{i}+h_{2}b_{4}^{i}\rightarrow
e_{3}^{i}, $ $\left( 1\leq i\leq m^{2}\right) $.

\qquad \qquad \qquad \qquad $h_{2}b_{3}^{i}+h_{1}b_{4}^{i}\rightarrow
e_{4}^{i}.$

\textbf{Step 4.} Solve $\det \left( H_{2}\right) \times
d_{i}=e_{4}^{i}(h_{1}b_{1}^{i}+h_{2}x_{i})-e_{3}^{i}(h_{2}b_{1}^{i}+h_{1}x_{i})
$.

\textbf{Step 5. }Substitute for $x_{i}=b_{2}^{i}$.

\textbf{Step 6.} Construct $B_{i}$.

\textbf{Step 7.} Construct $M$.

\textbf{Step 8.} End of algorithm.

We give following example as an application of the generalized Lucas
blocking algorithm for $b=1$.

\begin{example}
\label{exm2}Let us consider the message matrix for the following message
text:%
\begin{equation*}
\text{\textquotedblleft GOOD\textquotedblright }
\end{equation*}%
Using the message text, we get the following message matrix $M:$%
\begin{equation*}
M=\left[
\begin{array}{cc}
G & O \\
O & D%
\end{array}%
\right] _{2\times 2}.
\end{equation*}%
\textbf{Coding Algorithm:}

\textbf{Step 1. }We construct the message text $M$ of size $2\times 2$,
named $B_{1}:$%
\begin{equation*}
B_{1}=\left[
\begin{array}{cc}
G & O \\
O & D%
\end{array}%
\right] \text{.}
\end{equation*}%
\textbf{Step 2.} Since $b=1$, we calculate $n=2$. For $n=2$, we use the
following \textquotedblleft letter table\textquotedblright\ for the message
matrix $M:$%
\begin{equation*}
\begin{tabular}{|l|l|l|l|}
\hline
$G$ & $O$ & $O$ & $D$ \\ \hline
$8$ & $16$ & $16$ & $5$ \\ \hline
\end{tabular}%
.
\end{equation*}%
\textbf{Step 3.} We have the elements of the block $B_{1}$ as follows:%
\begin{equation*}
\begin{tabular}{|l|l|l|l|}
\hline
$b_{1}^{1}=8$ & $b_{2}^{1}=16$ & $b_{3}^{1}=16$ & $b_{4}^{1}=5$ \\ \hline
\end{tabular}%
.
\end{equation*}%
\textbf{Step 4.} Now we calculate the determinants $d_{1}$ of the block $%
B_{1}:$%
\begin{equation*}
\begin{tabular}{|l|}
\hline
$d_{1}=\det (B_{1})=-216$ \\ \hline
\end{tabular}%
.
\end{equation*}%
\textbf{Step 5.} Using Step 3 and Step 4 we obtain the following matrix $K:$%
\begin{equation*}
K=\left[
\begin{array}{cccc}
-216 & 8 & 16 & 5%
\end{array}%
\right] .
\end{equation*}%
\textbf{Step 6.} End of algorithm.\newline
\textbf{Decoding algorithm:}

\textbf{Step 1.} By $($\ref{eqn13}$),$ we know that%
\begin{equation*}
H_{2}=\left[
\begin{array}{cc}
1 & 3 \\
3 & 1%
\end{array}%
\right] \text{.}
\end{equation*}%
\textbf{Step 2. }The elements of $H_{2}$ are denoted by%
\begin{equation*}
h_{1}=1\text{ and }h_{2}=3\text{.}
\end{equation*}%
\textbf{Step 3.} We compute the elements $e_{3}^{1},e_{4}^{1}$ to construct
the matrix $E_{1}:$%
\begin{equation*}
e_{3}^{1}=31\text{, }e_{4}^{1}=53\text{.}
\end{equation*}%
\textbf{Step 4.} We calculate the elements $x_{1}:$%
\begin{eqnarray*}
\left( -8\right) \times \left( -216\right) &=&424+159x_{1}-744-31x_{1} \\
&\Rightarrow &x_{1}=16\text{.}
\end{eqnarray*}%
\textbf{Step 5.} We rename $x_{1}$ as follows$:$%
\begin{equation*}
x_{1}=b_{2}^{1}=16\text{.}
\end{equation*}%
\textbf{Step 6. }We construct the block matrix $B_{1}:$%
\begin{equation*}
B_{1}=\left[
\begin{array}{cc}
8 & 16 \\
16 & 5%
\end{array}%
\right] .
\end{equation*}%
\textbf{Step 7.} We obtain the message matrix $M:$%
\begin{equation*}
M=\left[
\begin{array}{cc}
8 & 16 \\
16 & 5%
\end{array}%
\right] =\left[
\begin{array}{cc}
G & O \\
O & D%
\end{array}%
\right] .
\end{equation*}%
\textbf{Step 8.} End of algorithm.
\end{example}

\section{Conclusions}

We have presented two new coding/decoding algorithms by means of the blocks
of sizes $3\times 3$ and $2\times 2$. Since the determinant of the matrix $%
G_{2}$ is $0$, we study the matrix $G_{n}$ for $n\geq 3$ in the generalized
Fibonacci blocking algorithm, although we can study with the matrix $H_{n}$
for $n\geq 2$ in the generalized Lucas blocking algorithm.

By differently taking $p$ and $q$, we can obtain different algorithms.
Furthermore it can be mixed the above new blocking methods with the previous
methods given in \cite{tas}, \cite{tas2} and \cite{ucar2}. It is possible to
produce new blocking methods similar to minesweeper algorithm given in \cite%
{ucar2}.

\end{document}